\documentclass{amsart}

\usepackage{amsmath,amssymb,amsthm}

\newtheorem{prop}{Proposition}
\newtheorem{thm}[prop]{Theorem}
\newtheorem{lem}[prop]{Lemma}

\theoremstyle{definition}

\newtheorem*{rem}{Remark}
\newtheorem*{ack}{Acknowledgements}

\def\co{\colon\thinspace}

\newcommand{\Q}{\mathbb{Q}}
\newcommand{\R}{\mathbb{R}}

\newcommand{\alst}{\alpha_{\mathrm{st}}}
\newcommand{\Rst}{R_{\mathrm{st}}}
\newcommand{\xist}{\xi_{\mathrm{st}}}

\begin{document}

\author[H.~Geiges]{Hansj\"org Geiges}
\address{Mathematisches Institut, Universit\"at zu K\"oln,
Weyertal 86--90, 50931 K\"oln, Germany}
\email{geiges@math.uni-koeln.de}

\author[N.~R\"ottgen]{Nena R\"ottgen}
\address{Mathematisches Institut, Albert-Ludwigs-Universit\"at Freiburg,
Eckerstr.~1, 79104 Freiburg, Germany}
\email{nena.roettgen@math.uni-freiburg.de}

\author[K.~Zehmisch]{Kai Zehmisch}
\address{Mathematisches Institut, Universit\"at zu K\"oln,
Weyertal 86--90, 50931 K\"oln, Germany}
\email{kai.zehmisch@math.uni-koeln.de}

\title[Trapped Reeb orbits]{Trapped Reeb orbits do not imply periodic ones}

\date{}

\begin{abstract}
We construct a contact form on $\R^{2n+1}$, $n\geq 2$, equal to
the standard contact form outside a compact set
and defining the standard contact structure on all
of $\R^{2n+1}$, which has
trapped Reeb orbits, including a torus invariant under the Reeb flow,
but no closed Reeb orbits. This answers a question posed by Helmut Hofer.
\end{abstract}


\maketitle


\section{Introduction}
In \cite[Theorem~2]{elho94}, Eliashberg and Hofer proved a global version
of the Darboux theorem for contact forms in dimension~$3$:
Any contact form $\alpha$ on $\R^3$ that equals the standard form
\[ \alst=dz+\frac{1}{2}(x\, dy-y\, dx)\]
outside a compact set and whose Reeb vector field does not have
any periodic orbits, is diffeomorphic to the standard form, i.e.\
there is a diffeomorphism $\phi$ of $\R^3$ such that $\phi^*\alpha=\alst$.

Recall that a contact form $\alpha$ on a $(2n+1)$-dimensional manifold
is a $1$-form such that $\alpha\wedge (d\alpha)^n$ is a volume form.
The Reeb vector field of such a contact form is the unique vector
field $R$ satisfying
\[ d\alpha(R,\, .\,)\equiv 0\;\;\;\text{and}\;\;\;\alpha(R)\equiv 1.\]

These defining equations imply that diffeomorphic contact forms
have diffeomorphic Reeb vector fields, so if $\phi^*\alpha=\alst$,
then $T\phi(\Rst)=R$, where $\Rst=\partial_z$ is the
Reeb vector field of~$\alst$. Thus, the Reeb vector field of a contact
form $\alpha$ on $\R^3$ satisfying the assumptions of the
Eliashberg--Hofer theorem does not have any orbits that
are bounded in forward or backward time (we shall call such
orbits `trapped'). Phrased contrapositively:

\begin{thm}[Eliashberg--Hofer]
\label{thm:EH}
Let $\alpha$ be a contact form on $\R^3$ that equals the standard form
$\alst$ outside a compact set. If the Reeb vector field of $\alpha$
has a trapped orbit, then it also has a periodic orbit.\qed
\end{thm}

By taking the connected sum of $(\R^3,\alst)$ with a $3$-sphere
carrying the standard contact form (all of whose Reeb orbits
are closed), one can easily construct a contact form on $\R^3$
that equals $\alst$ outside a compact set but has periodic Reeb orbits
(and hence cannot be diffeomorphic to~$\alst$).

In a talk at the conference on \emph{Recent Progress in Lagrangian
and Hamiltonian Dynamics} (Lyon, 2012) and in personal communication
to Victor Bangert, Helmut Hofer conjectured the higher-dimensional
analogue of Theorem~\ref{thm:EH}, see also~\cite{brho11}.
The purpose of this note is to disprove
that conjecture by an example.

We write
\[ \alst=dz+\frac{1}{2}\sum_{j=1}^n (x_j\, dy_j-y_j\, dx_j)\]
for the standard contact form on $\R^{2n+1}$, and $\xist=\ker\alst$
for the standard contact structure.

\begin{thm}
\label{thm:main}
There is a contact form $\alpha$ on $\R^{2n+1}$, $n\geq 2$,
defining the standard contact structure, i.e.\ $\ker\alpha=\xist$,
with the following properties:
\begin{enumerate}
\item[(i)] The Reeb vector field $R$ of $\alpha$ has a compact
invariant set (and hence orbits bounded in forward and backward time).
\item[(ii)] There are Reeb orbits which are bounded in
forward time and whose $z$-component goes to
$-\infty$ for $t\rightarrow -\infty$.
\item[(iii)] $\alpha$ equals $\alst$ outside a compact set.
\item[(iv)] $R$ does not have any periodic orbits.
\end{enumerate}
\end{thm}

A related result in Riemannian geometry is due to
Bangert and the second author. In \cite{baro12}, answering a question
of Walter Craig, they showed the existence of a Riemannian
metric on $\R^n$, $n\geq 4$, equal to the Euclidean metric outside
a compact set, that admits bounded geodesics 
(or `trapped bicharacteristics') but no periodic ones.

A contact form with the Reeb dynamics described in
Theorem~\ref{thm:main} was first discovered by the second
author~\cite{roet13}. In joint work we derived the simple construction
of such an example that we are going to present now.
\section{Reeb and contact vector fields}
Let $(M,\xi=\ker\alpha)$ be a contact manifold.
A contact vector field is a vector field whose flow preserves
the contact structure~$\xi$. Once a contact form
$\alpha$ has been chosen, there is a one-to-one
correspondence between smooth functions $H\co  M\rightarrow\R$
and contact vector fields~$X$, defined as follows
(cf.~\cite[Theorem~2.3.1]{geig08}): Given $H$, the corresponding
contact vector field $X$ is given by $X=HR+Y$, where $R$
is the Reeb vector field of $\alpha$ and $Y$ is the
unique vector field tangent to $\xi$ satisfying
\begin{equation}
\label{eqn:Y}
i_Yd\alpha=dH(R)\alpha-dH.
\end{equation}
Conversely, the Hamiltonian function $H$ corresponding to
a contact vector field $X$ is given by $H=\alpha(X)$.

The Reeb vector field~$R$, corresponding to the constant function~$1$,
is a contact vector field whose flow even preserves the contact form~$\alpha$.
The following well-known lemma says that any contact vector field
positively transverse to $\xi$ is the Reeb vector field of some
contact form for~$\xi$. The proof is a straightforward
computation using the defining equations of the Reeb vector field.

\begin{lem}
\label{lem:Reeb}
The contact vector field corresponding to the
positive Hamiltonian function $H\co M\rightarrow\R^+$ is
the Reeb vector field of the contact form~$\alpha/H$.\qed
\end{lem}
\section{The example}
We are going to prove Theorem~\ref{thm:main} for $n=2$;
the higher-dimensional generalisation is straightforward.
Thus, $\alst$ now denotes the standard contact form on $\R^5$,
with Reeb vector field $\Rst=\partial_z$. Write $(r_j,\theta_j)$
for the polar coordinates in the $(x_j,y_j)$-plane,
$j=1,2$. By Lemma~\ref{lem:Reeb} it suffices
to construct a contact vector field positively transverse to~$\xist$
with the desired dynamics.

\begin{prop}
\label{prop:X}
There is a contact vector field $X$ for $\xist$ with the
following properties:
\begin{enumerate}
\item[(X-i)] On the Clifford torus
\[ T:=\{r_1=1,\, r_2=1,\, z=0\} \]
the vector field $X$ equals $\partial_{\theta_1}+s\partial_{\theta_2}$
for some $s\in [0,1]\setminus\Q$.
\item[(X-ii)] The cylinder $T\times [-1,0]$, i.e.\
\[ \{ r_1=1,\, r_2=1,\, z\in[-1,0]\},\]
is mapped to itself under the flow of $X$ in forward time.
\item[(X-iii)] Outside a compact neighbourhood of~$T$, the vector field
$X$ equals~$\partial_z$.
\item[(X-iv)] On $\R^5\setminus T$ we have $dz(X)>0$.
\end{enumerate}
\end{prop}

Condition (X-i) guarantees that the Clifford torus $T$ is an
invariant set of $X$ without any closed orbits. Then by condition
(X-iv) there are no closed orbits whatsoever. Condition (X-iii)
ensures that the contact form with Reeb vector field $X$ is
the standard form $\alst$ outside a compact neighbourhood of~$T$.
With condition (X-ii) this yields an orbit coming from $-\infty$
and trapped in forward time, since $T$ is attracting for the
whole cylinder $T\times [-1,0]$. Likewise, our construction will
yield orbits trapped in backward time and going off to~$\infty$.

\begin{proof}[Proof of Proposition~\ref{prop:X}]
We wish to construct $X$ as the contact vector field
corresponding to a Hamiltonian function $H\co\R^5\rightarrow\R^+$.
To that end, we translate the conditions on~$X$ into conditions
on~$H$.

With $dH(\Rst)=H_z$, equation (\ref{eqn:Y}) for $\alpha=\alst$
becomes
\begin{equation}
\label{eqn:Yda}
i_Yd\alst=\sum_{j=1}^2\left(-\bigl(\frac{y_j}{2}H_z+H_{x_j}\bigr)\, dx_j+
\bigl(\frac{x_j}{2}H_z-H_{y_j}\bigr)\, dy_j\right).
\end{equation}
The contact structure $\xist$ is spanned by the vector fields
\[ e_j=\partial_{x_j}+\frac{y_j}{2}\partial_z,\;\;\;
f_j=\partial_{y_j}-\frac{x_j}{2}\partial_z,\;\;\; j=1,2.\]
By writing $Y$ in terms of these vector fields, we find with
equation~(\ref{eqn:Yda}) that
\begin{equation}
\label{eqn:Yst}
Y=\sum_{j=1}^2\left(\bigl(\frac{x_j}{2}H_z-H_{y_j}\bigr)\, e_j
+\bigl(\frac{y_j}{2}H_z+H_{x_j}\bigr)\, f_j\right).
\end{equation}

Condition (X-i) says that along $T$ we must have
\[ H=\alst(\partial_{\theta_1}+s\partial_{\theta_2})=\frac{1+s}{2}\]
and
\[ Y=X-H\Rst =
\partial_{\theta_1}+s\partial_{\theta_2}-\frac{1+s}{2}\partial_z.\]
With (\ref{eqn:Yst}) this gives
\[ \left. \begin{array}{rcl}
H_{x_1} & = & x_1-\frac{y_1}{2}H_z\\[1mm]
H_{y_1} & = & y_1+\frac{x_1}{2}H_z\\[1mm]
H_{x_2} & = & sx_2-\frac{y_2}{2}H_z\\[1mm]
H_{y_2} & = & sy_2+\frac{x_2}{2}H_z
\end{array}\right\}\text{on $T$.}\]
But on $T$ we also have
\[ 0=dH(\partial_{\theta_j})=x_jH_{y_j}-y_jH_{x_j},\]
which by the previous equations equals $H_z/2$. So in fact we obtain
\begin{equation}
\tag{H-i}
\left. \begin{array}{lcl}
H       & = & (1+s)/2\\
H_{x_1} & = & x_1\\
H_{y_1} & = & y_1\\
H_{x_2} & = & sx_2\\
H_{y_2} & = & sy_2\\
H_z     & = & 0
\end{array}\right\}\text{on $T$.}
\end{equation}

Next we turn to condition (X-ii). For the moment we may disregard the
$\partial_z$-component of $X$, as this will be controlled by the
condition on $H$ corresponding to (X-iv). By looking at
equation~(\ref{eqn:Yst})
we see that $X$ will have the required behaviour (and the similar one for
the flow on $T\times [0,1]$ in backward time) if we stipulate
\begin{equation}
\tag{H-ii}
\text{$H=(1+s)/2$ on the cylinder $\{ r_1=1,\, r_2=1,\, z\in[-1,1]\}$.}
\end{equation}
Indeed, then $H_z=0$ on that cylinder, and
\[ 0=H_{\theta_j}=x_jH_{y_j}-y_jH_{x_j},\; j=1,2,\]
which implies that $H_{x_j}\partial_{y_j}-H_{y_j}\partial_{x_j}$
is proportional to $x_j\partial_{y_j}-y_j\partial_{x_j}=\partial_{\theta_j}$
on that cylinder.

Condition (X-iii) simply translates into
\begin{equation}
\tag{H-iii}
\text{$H\equiv 1$ outside a compact neighbourhood of $T$.}
\end{equation}

Finally, from (\ref{eqn:Yst}) we find that
\[ dz(Y)=-\frac{1}{2}\sum_{j=1}^2\left(x_jH_{x_j}+y_jH_{y_j}\right),\]
so condition (X-iv) is equivalent to
\begin{equation}
\tag{H-iv}
H-\frac{1}{2}\sum_{j=1}^2\left(x_jH_{x_j}+y_jH_{y_j}\right)>0\;\;
\text{on $\R^5\setminus T$.}
\end{equation}

We now proceed to construct an explicit function $H$ satisfying
properties (H-i) to (H-iv). The basic idea is very simple.
We modify the function
\[ (x_1,y_1,x_2,y_2,z)\longmapsto \frac{1}{2}(x_1^2+y_1^2)+
\frac{s}{2}(x_2^2+y_2^2),\]
which satisfies (H-i), such that conditions
(H-ii) to (H-iv) are also satisfied. This essentially
amounts to smoothing out this function in such a way that it
becomes constant 1 outside a compact neighbourhood
of $T$, and such that it has a growth rate in radial direction
in the planes $\{z=\text{const.}\}$ smaller than the
quadratic growth rate of the function we start with.

Let $f_z\co\R_0^+\rightarrow\R$,
$z\in\R$, be a smooth family of smooth functions
with the following properties:
\begin{enumerate}
\item[(i)] $f_z(1)=0$ for all~$z$;
\item[(ii)] $tf_z'(t)\leq 1$ for all $z$ and $t$, with equality
only for $z=0$ and $t=1$;
\item[(iii)] for $t$ large (uniformly in~$z$),
$f_z(t)>\log c$ for some constant $c>2/s>2$
\end{enumerate}
In other words, $f_z$ has the same value as $\log$
at $t=1$, $f_0$ has the same derivative at $t=1$ as $\log$, 
for other values of $z$ or $t$ the function $f_z$
grows more slowly than $\log$. The function
\[ H_0(x_1,y_1,x_2,y_2,z):=
\frac{1}{2}\exp \bigl(f_z(x_1^2+y_1^2)\bigr)+
\frac{s}{2}\exp \bigl(f_z(x_2^2+y_2^2)\bigr)\]
satisfies (H-i) and (H-iv), and it satisfies
(H-ii) on the whole cylinder (in $z$-direction) over~$T$.

Notice that by condition (iii) on $f_z$, either of the
summands in $H_0$ is greater than $sc/2>1$ for $r_1$ resp.\ $r_2$
sufficiently large. This will be used below when we enforce
condition (H-iii).

Let $g\co\R^+\rightarrow\R$
be a smooth monotone increasing function with these properties:
\begin{enumerate}
\item[(i)] $g(t)=\log t$ near $t=(1+s)/2$;
\item[(ii)] $g(t)=0$ for $t\geq sc/2$;
\item[(iii)] $g'(t)\leq 1/t$ for all $t$.
\end{enumerate}
Then $H_1:=\exp(g\circ H_0)$ satisfies all requirements
bar one: (H-iii) only holds outside a cylinder over a
compact neighbourhood of $T$ in $\{ z=0\}$.

Finally, we choose a smooth function $h\co\R\rightarrow [0,1]$ with
\begin{enumerate}
\item[(i)] $h(z)=0$ for $z\in [-1,1]$;
\item[(ii)] $h(z)=1$ for $|z|$ large.
\end{enumerate}
Then set
\[ H(x_1,y_1,x_2,y_2,z)=(1-h(z))\cdot H_1(x_1,y_1,x_2,y_2,z)+h(z).\]
This positive function $H$ satisfies conditions (H-i) to (H-iv).
\end{proof}

\begin{rem}
Statement (ii) in Theorem~\ref{thm:main} is a topological consequence of
statements (i) and~(iii): Consider a hyperplane $E=\{ z=-z_0\}$ with
$z_0>0$ sufficiently large, such that $R=\partial_z$ along~$E$.
The flow of $R$ (for any given finite time) cannot send $E$
to the region $\{z>0\}$, since this is obstructed by the invariant
torus~$T$. Our proof, in addition, gives explicit orbits trapped
in one direction of time only.
\end{rem}
\begin{ack}
We thank Victor Bangert for directing our attention to this question.
This note was written during the workshop on Legendrian submanifolds,
holomorphic curves and generating families at the Acad\'emie Royale
de Belgique, August 2013, organised by Fr\'ed\'eric Bourgeois.
H.~G.\ and K.~Z.\ are partially supported by DFG grants GE 1245/2-1
and ZE 992/1-1, respectively.
\end{ack}

\end{document}